\documentclass[a4paper,11pt]{article}
\usepackage{amsthm,thmtools}
\usepackage{amsmath}
\usepackage{mathtools}
\usepackage{amssymb}
\usepackage{amsfonts}
\usepackage[top=80pt, bottom=90pt, left=70pt, right=70pt]{geometry}
\usepackage{comment}
\usepackage{color}
\usepackage[colorlinks=true,linkcolor=blue,citecolor=blue]{hyperref}
\usepackage{cases}
\usepackage{enumitem}
\usepackage{xspace}
\usepackage{fancybox}
\usepackage{environ}
\usepackage{complexity}
\usepackage{cleveref}
\usepackage{tabularx}
\usepackage{graphicx} 

\usepackage[bibliography=common]{apxproof}

\newcounter{nodecount}

\declaretheorem[style=plain,name=Theorem]{theorem}
\declaretheorem[style=plain,sibling=theorem,name=Lemma]{lemma}

\declaretheorem[style=plain,sibling=theorem,name=Claim]{claim}
\declaretheorem[style=plain,sibling=theorem,name=Observation]{observation}

\crefname{theorem}{Theorem}{Theorems}
\crefname{proposition}{Proposition}{Propositions}
\crefname{lemma}{Lemma}{Lemmas}
\crefname{exmp}{Example}{Examples}
\crefname{corollary}{Corollary}{Corollarys}
\crefname{claim}{Claim}{Claims}
\crefname{remark}{Remark}{Remarks}
\crefname{section}{Section}{Sections}

\usepackage[linesnumbered,noend,ruled,vlined]{algorithm2e}
\SetKwInput{KwInput}{Input}     
\SetKwInput{KwOutput}{Output}   

\newcommand{\OPT}{\operatorname{\textup{OPT}}}

\newcommand{\ourprob}{Basis Sequence Reconfiguration}
\newcommand{\prb}[1]{\textnormal{\scshape #1}}

\newcommand{\symdif}{\mathbin{\triangle}}
\newcommand{\rseq}[1]{\langle#1\rangle}
\newcommand{\tail}[1]{\operatorname{\mathrm{tail}}(#1)}
\newcommand{\head}[1]{\operatorname{\mathrm{head}}(#1)}
\newcommand{\elmmat}[2]{M_{#1}^{#2}}
\newcommand{\elmver}[2]{e_{#1}^{#2}}
\newcommand{\elmgrd}[2]{E_{#1}^{#2}}
\newcommand{\cntver}[2]{c_{#1}^{#2}}
\newcommand{\setmat}[2]{M_{#1}^{#2}}
\newcommand{\setver}[2]{s_{#1}^{#2}}
\newcommand{\setgrd}[2]{E_{#1}^{#2}}
\newcommand{\floor}[1]{\left\lfloor {#1} \right\rfloor }
\newcommand{\source}[1]{{#1}^\mathtt{s}}
\newcommand{\sink}[1]{{#1}^\mathtt{t}}
\newcommand{\elemid}[2]{f({#1},{#2})}
\newcommand{\Llength}{n^2}

\newcommand{\rev}[1]{#1}

\makeatletter
\newcommand{\problemtitle}[1]{\gdef\@problemtitle{#1}}
\newcommand{\probleminput}[1]{\gdef\@probleminput{#1}}
\newcommand{\problemoutput}[1]{\gdef\@problemoutput{#1}}
\NewEnviron{problem}{
  \problemtitle{}\probleminput{}\problemoutput{}
  \BODY
  \par\addvspace{.5\baselineskip}
  \noindent{
    \framebox[\textwidth][c]{
        \begin{tabularx}{\textwidth}{@{\hspace{\parindent}} l X c}
            \multicolumn{2}{@{\hspace{\parindent}}l}{\textsc{\@problemtitle}} \\
            \textbf{Input:} & \@probleminput \\
            \textbf{Question:} & \@problemoutput
        \end{tabularx}
        }
      }
  \par\addvspace{.5\baselineskip}
}
\makeatother

\makeatletter

\@addtoreset{equation}{section}
\makeatother

\date{}

\begin{document}

\title{Basis sequence reconfiguration in the union of matroids}
\author{Tesshu Hanaka\thanks{Faculty of Information Science and Electrical Engineering, Kyushu University, Japan.
      Email: \texttt{hanaka@inf.kyushu-u.ac.jp}}
\and Yuni Iwamasa\thanks{Graduate School of Informatics, Kyoto University, Japan.
    Email: \texttt{iwamasa@i.kyoto-u.ac.jp}}
\and Yasuaki Kobayashi\thanks{Faculty of Information Science and Technology, Hokkaido University, Japan.
    Email: \texttt{koba@ist.hokudai.ac.jp}}
\and Yuto Okada\thanks{Graduate School of Informatics, Nagoya University, Japan.
    Email: \texttt{okada.yuto.b3@s.mail.nagoya-u.ac.jp}}
\and Rin Saito\thanks{Graduate School of Information Sciences, Tohoku University, Japan.
        Email: \texttt{rin.saito@dc.tohoku.ac.jp}}}

\maketitle

\begin{abstract}
    Given a graph $G$ and two spanning trees $T$ and $T'$ in $G$, {\scshape Spanning Tree Reconfiguration} asks whether there is a step-by-step transformation from $T$ to $T'$ such that all intermediates are also spanning trees of $G$, by exchanging an edge in $T$ with an edge outside $T$ at a single step.
    This problem is naturally related to matroid theory, which shows that there always exists such a transformation for any pair of $T$ and $T'$.
    Motivated by this example, we study the problem of transforming a sequence of spanning trees into another sequence of spanning trees.
    We formulate this problem in the language of matroid theory: Given two sequences of bases of matroids, the goal is to decide whether there is a transformation between these sequences.
    We design a polynomial-time algorithm for this problem, even if the matroids are given as basis oracles.
    To complement this algorithmic result, we show that the problem of finding a shortest transformation is NP-hard to approximate within a factor of $c \log n$ for some constant $c > 0$, where $n$ is the total size of the ground sets of the input matroids.
    \begin{description}
    \item[Keywords] Combinatorial reconfiguration, Matroids, Polynomial-time algorithm, Inapproximability
    \end{description}
\end{abstract}

\section{Introduction}

In \emph{reconfiguration problems} (see~\cite{Heuvel13:survey,Nishimura18:survey} for introductory material), given two (feasible) configurations in a certain system, the objective is to determine whether there exists a step-by-step transformation between these configurations such that all intermediate configurations are also feasible.
Among numerous reconfiguration problems studied in the literature, one of the first problems explicitly recognized as a reconfiguration problem is \prb{Spanning Tree Reconfiguration}.
In this problem, given two spanning trees $T, T'$ in a (multi)graph $G$, one is asked to find a transformation from one spanning tree $T$ into the other spanning tree $T'$ by repeatedly exchanging a single edge (i.e., $T - e + f$ for an edge $e \in E(T)$ and $f \in E(G) \setminus E(T)$), such that all intermediates are also spanning trees in $G$.
Ito et al.~\cite{ItoDHPSUU11:TCS:complexity}\footnote{More specifically, they considered a weighted version of this problem.} observed that one can always find such a transformation with exactly $|E(T) \setminus E(T')|$ exchanges by exploiting a well-known property of \emph{matroids}.

Let $E$ be a finite set and let $\mathcal B \subseteq 2^E$ be a nonempty collection of subsets of $E$ that satisfies the following \emph{basis exchange axiom}: for distinct $B, B' \in \mathcal B$ and $x \in B \setminus B'$, there is $y \in B' \setminus B$ satisfying $B - x + y \in \mathcal B$.
Then, the pair $M = (E, \mathcal B)$ is called a \emph{matroid}, and each set in $\mathcal B$ is called a \emph{basis} of $M$.
For a connected graph $G$ with edge set $E(G)$, let $\mathcal T$ be the collection of all edge subsets, each of which induces a spanning tree in $G$.
Then it is well known that $\mathcal{T}$ satisfies the basis exchange axiom: for each $e \in E(T) \setminus E(T')$, there is an edge $f \in E(T') \setminus E(T)$ such that $T - e + f$ is a spanning tree of $G$.
Hence, the pair $(E(G), \mathcal T)$ is a matroid, called a \emph{graphic matroid}. 
This also allows us to find a transformation of $|E(T) \setminus E(T')|$ exchanges for \prb{Spanning Tree Reconfiguration}.
Since every transformation between $T$ and $T'$ requires at least $|E(T) \setminus E(T')|$ exchanges, this is a shortest one among all transformations.

In this paper, we address a natural extension of \prb{Spanning Tree Reconfiguration}.
Let $G$ be a (multi)graph.
We say that a sequence of $k$ spanning trees $(T_1, \dots, T_k)$ of $G$ is \emph{feasible} if the spanning trees are edge-disjoint.
A pair of two feasible sequences of spanning trees $\mathbb T  = (T_1, \dots, T_k)$ and $\mathbb T' = (T'_1, \dots, T'_k)$ is said to be \emph{adjacent} if there is an index $1 \le i \le k$ such that $T_j = T'_j$ for $1 \le j \le k$ with $i \neq j$ and $T'_i = T_i - e + f$ for some $e \in E(T_i)$ and $f \in E(G) \setminus E(T_i)$.
Given two feasible sequences of $k$ spanning trees $\mathbb T = (T_1, \dots, T_k)$ and $\mathbb T' = (T'_1, \dots, T'_k)$ of a graph $G = (V, E)$, \prb{Spanning Tree Sequence Reconfiguration} asks whether there are feasible sequences $\mathbb T_0, \dots, \mathbb T_\ell$ such that $\mathbb T_0 = \mathbb T$, $\mathbb T_\ell = \mathbb T'$, and $\mathbb T_{i - 1}$ and $\mathbb T_i$ are adjacent for all $1 \le i \le \ell$.
This type of problem naturally extends conventional reconfiguration problems by enabling a ``simultaneous transformation'' of multiple mutually exclusive solutions.

To address \prb{Spanning Tree Sequence Reconfiguration}, we consider a more general problem, called \textsc{\ourprob}.
Let $\mathbb M = (M_1, \dots, M_k)$ be a sequence of matroids, where $M_i = (E_i, \mathcal B_i)$ for $1 \le i \le k$.
Let us note that $E_i$ and $E_j$ may not be disjoint for distinct $i$ and $j$.
A \emph{basis sequence} of $\mathbb M$ is a sequence $\mathbb B = (B_1, \dots, B_k)$ such that $B_i$ is a basis of $M_i$ (i.e., $B_i \in \mathcal B_i$).
A basis sequence $\mathbb B = (B_1, \dots, B_k)$ is said to be \emph{feasible} for $\mathbb M$ if $B_i \cap B_j = \emptyset$ for $1 \le i < j \le k$.
A pair of feasible basis sequences $\mathbb B = (B_1, \dots, B_k)$ and $\mathbb B' = (B'_1, \dots, B'_k)$ is said to be \emph{adjacent} if there is an index $1 \le i \le k$ such that $B_j = B'_j$ for $1 \le j \le k$ with $i \neq j$ and $B'_i = B_i - x + y$ for some $x \in B_i$ and $y \in E_i \setminus B_i$.
A feasible basis sequence $\mathbb B$ is \emph{reconfigurable to} a feasible basis sequence $\mathbb B'$ if there are feasible basis sequences $\mathbb B_0, \dots, \mathbb B_\ell$ of $\mathbb M$ such that $\mathbb B_0 = \mathbb B$, $\mathbb B_\ell = \mathbb B'$, and $\mathbb B_{i-1}$ and $\mathbb B_{i}$ are adjacent for all $1 \le i \le \ell$.
We refer to such a sequence $\langle \mathbb B_0, \dots, \mathbb B_\ell \rangle$ as a \emph{reconfiguration sequence} between $\mathbb B$ and $\mathbb B'$.
Our problem is formally defined as follows.
\begin{problem}
\problemtitle{\ourprob}
\probleminput{A tuple $\mathbb M = (M_1, \dots, M_k)$ of $k$ matroids and feasible basis sequences $\mathbb{B} = (B_1, \ldots, B_k)$ and $\mathbb{B}' = (B_1', \ldots, B_k')$.
}
\problemoutput{Determine if $\mathbb{B}$ is reconfigurable to $\mathbb{B}'$.}
\end{problem}

Note that if $M_i = (E(G),\mathcal{T})$  for every $i$, $\mathbb{B}=\mathbb{T}$, and $\mathbb{B}'=\mathbb{T}'$, \textsc{\ourprob} is equivalent to \prb{Spanning Tree Sequence Reconfiguration}. 

We also consider an optimization variant of \prb{\ourprob}: Given an instance of \prb{\ourprob}, the goal is to find a shortest reconfiguration sequence between $\mathbb B$ and $\mathbb B'$.
We refer to this problem as \prb{Shortest \ourprob}.

We investigate the computational complexity of \textsc{\ourprob}. 
In this paper, matroids are sometimes given as \emph{basis oracles}, that is, given a set $X \subseteq E$ of a matroid $M = (E, \mathcal B)$, the basis oracle (of $M$) returns true if and only if $X \in \mathcal B$.
In such a case, we can access $\mathcal B$ through this oracle and assume that the basis oracle can be evaluated in polynomial in $|E|$. 
Our main contribution is as follows.

\begin{theorem} \label{thm:solvability}
    \prb{\ourprob} can be solved in polynomial time, assuming that the input matroids are given as basis oracles.
    Moreover, if the answer is affirmative, we can compute a reconfiguration sequence between given two feasible basis sequences in polynomial time as well.
\end{theorem}

This result nontrivially generalizes the previous result of \cite{ItoDHPSUU11:TCS:complexity}.
It would be worth mentioning that, in contrast to \prb{Spanning Tree Reconfiguration}, our problem \prb{Spanning Tree Sequence Reconfiguration} has infinitely many no-instances (see~\Cref{fig:no-instance} for an example).
\begin{figure}
    \centering
    \includegraphics[width=0.7\textwidth]{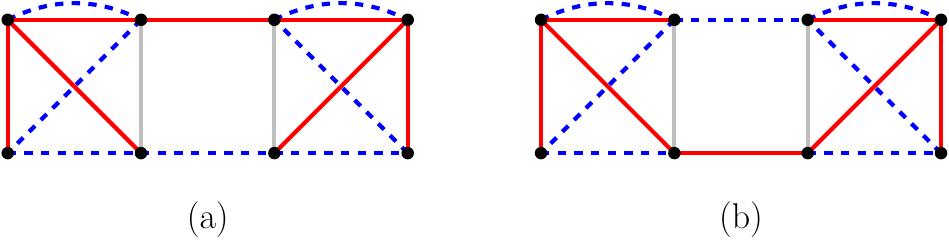}
    \caption{The figure illustrates an instance in which a pair of edge-disjoint spanning trees~(a) cannot be transformed into the other pair~(b), where the spanning trees are indicated by dashed blue lines and solid red lines.}
    \label{fig:no-instance}
\end{figure}

A natural extension of \prb{\ourprob} is to find a \emph{shortest} reconfiguration sequence.
Unfortunately, we show that it is hard to find it in polynomial time, even for approximately shortest reconfiguration sequences. 

\begin{theorem} \label{thm:hardness}
    \prb{Shortest \ourprob} is $\NP$-hard even if the input sequence $\mathbb M$ consists of two partition matroids.
    Furthermore, unless $\P = \NP$, \prb{Shortest \ourprob} cannot be approximated in polynomial time within a factor of $c\log n$ for some constant $c > 0$, where $n$ is the total size of the ground sets of the input matroids.
\end{theorem}

\paragraph*{Related work.}
Due to the property of ``one-by-one exchange'' in combinatorial reconfiguration, various reconfiguration problems are naturally related to matroids~\cite{BercziS03:arXiv:Reconfiguration,BousquetHKMS23:FVSreconf,ItoDHPSUU11:TCS:complexity,ItoIKKKMS23:time-respecting-arbo,ItoIKNOW23:arborescences,KobayashiMS23:Arborescences}.
As mentioned above, Ito et al.~\cite{ItoDHPSUU11:TCS:complexity} studied \prb{Spanning Tree Reconfiguration} and showed that every spanning tree can be transformed into any other spanning tree in a graph.
Given this fact, Ito et al.~\cite{ItoIKNOW23:arborescences} further considered a directed analogue of this problem, in which the objective is to determine whether two arborescences (i.e., directed spanning trees) in a directed graph are transformed into each other.
Contrary to the undirected counterpart, for a (weakly) connected directed graph $D = (V, A)$, the pair $(A, \mathcal F)$ is not a matroid in general, where
$\mathcal{F}$ denotes the family of arc sets $F \subseteq A$, each of which forms an arborescence of $D$, 
while it is the collection of common bases of two matroids, i.e., $\mathcal F = \mathcal B_1 \cap \mathcal B_2$ for some matroids $(A, \mathcal B_1)$ and $(A, \mathcal B_2)$.
They still showed that every arborescence can be transformed into any other arborescence in a directed graph.
As a generalization of~\cite{ItoIKNOW23:arborescences},
Kobayashi, Mahara, and Schwarcz~\cite{KobayashiMS23:Arborescences} studied the reconfiguration problem of (not the \emph{sequence} of but) the \emph{union} of disjoint arborescences.
Namely, in their setting, a feasible solution is
the union $\bigcup_{i = 1}^k F_i$ of disjoint arborescences $F_1, F_2, \dots, F_k$,
and two feasible solutions $\bigcup_{i = 1}^k F_i$ and $\bigcup_{i = 1}^k F_i'$
are adjacent if and only if
there are $x \in \bigcup_{i = 1}^k F_i \setminus \bigcup_{i = 1}^k F_i'$ and $y \in \bigcup_{i = 1}^k F_i' \setminus \bigcup_{i = 1}^k F_i$
such that $\bigcup_{i = 1}^k F_i - x + y = \bigcup_{i = 1}^k F_i'$.
We note that even if two feasible solutions $\bigcup_{i = 1}^k F_i$ and $\bigcup_{i = 1}^k F_i'$
are adjacent in the sense of~\cite{KobayashiMS23:Arborescences},
the corresponding tuples $(F_1, F_2, \dots, F_k)$ and $(F_1', F_2', \dots, F_k')$ may not be adjacent in our sense.
It is worth mentioning that the reconfiguration problem of the \emph{union} of disjoint bases is trivially solvable, since it is just the reconfiguration problem of bases of the union of matroids; see \Cref{sec:preliminaries} for the definition of the matroid union.
For other reconfiguration problems related to (common bases of) matroids, see~\cite{BousquetHKMS23:FVSreconf,ItoIKNOW23:arborescences}.

Our work is highly related to a recent work of B\'erczi, M\'atrav\"olgyi, and Schwarcz~\cite{BercziS03:arXiv:Reconfiguration}.
They considered the symmetric exchange version of our problem, where two (not necessarily feasible) basis sequences $\mathbb B = (B_1, \dots, B_k)$ and $\mathbb B' = (B'_1, \dots, B'_k)$ are adjacent if there are $x \in B_i \setminus B_j$ and $y \in B_j \setminus B_i$ such that
\begin{align*}
    \mathbb B' = (B_1, \dots, B_{i-1}, B_i - x + y, B_{i + 1}, \dots, B_{j-1}, B_j - y + x, B_{j+1}, \dots, B_k).
\end{align*}
This reconfiguration problem has received considerable attention as its reconfigurability is essentially equivalent to White's conjecture~\cite{WHITE198081}.
(See~\cite{BercziS03:arXiv:Reconfiguration} for a comprehensive overview of White's conjecture.)
In particular, the conjecture states that for any pair of two feasible basis sequences $\mathbb B = (B_1, \dots, B_k)$ and $\mathbb B' =(B'_1, \dots, B'_k)$, $\mathbb B$ is reconfigurable to $\mathbb B'$ (by symmetric exchanges) if and only if $\bigcup_{i=1}^k B_i = \bigcup_{i=1}^k B'_i$.
The conjecture is confirmed for graphic matroids~\cite{Blasiak08:toric-ideal,FarberRS85:edge-disjoint}, which means that for every pair of sequences of edge-disjoint $k$ spanning trees $(T_1, \dots, T_k)$ and $(T'_1, \dots, T'_k)$ in a graph, one is reconfigurable to the other by symmetric exchanges if $\bigcup_{i=1}^kE(T_i) = \bigcup_{i=1}^kE(T'_i)$. This is in contrast to our setting, having an impossible case as seen in~\Cref{fig:no-instance}.

We would like to emphasize that our setting is also quite natural as it can be seen as a reconfiguration problem in the \emph{token jumping model}, which is best studied in the context of combinatorial reconfiguration~\cite{Heuvel13:survey,Nishimura18:survey}.
In particular, our problem can be regarded as a reconfiguration problem for \emph{multiple} solutions.
One of the most well-studied problems in this context is \prb{Coloring Reconfiguration}~\cite{BonamyB18:coloring,BonsmaC09:TCS:coloring,HatanakaIZ19:coloring}, which can be seen as a multiple solution variant of \prb{Independent Set Reconfiguration}.
There are several results working on reconfiguration problems for multiple solutions, such as \prb{Disjoint Paths Reconfiguration}~\cite{ItoIK0MNOO23:ReroutingDP} and \prb{Disjoint Shortest Paths Reconfiguration}~\cite{SaitoEIU23:VDSPreconf}.

\section{Preliminaries}\label{sec:preliminaries}
For a positive integer $n$, let $[n] \coloneqq \{1, 2, \ldots, n\}$.
For integers $p$ and $q$ with $p \leq q$, let $[p, q] \coloneqq \{p, p + 1, \ldots, q - 1, q\}$. 
For sets $X$ and $Y$, the \emph{symmetric difference} of $X$ and $Y$ is defined as $X \symdif Y \coloneqq (X \setminus Y) \cup (Y \setminus X)$.

Let $E$ be a finite set and let $\mathcal B \subseteq 2^E$ be a nonempty collection of subsets of $E$.
We say that $M = (E, \mathcal B)$ is a \emph{matroid} if for $B, B' \in \mathcal B$ and $x \in B \setminus B'$, there is $y \in B' \setminus B$ satisfying $(B \setminus \{x\}) \cup \{y\} \in \mathcal B$.
For notational convenience, we may write $B - x + y$ instead of $(B \setminus \{x\}) \cup \{y\}$.
Each set in $\mathcal B$ is called a \emph{basis} of $M$.
It is easy to verify that each basis of $M$ has the same cardinality, which is called the \emph{rank} of $M$.
In this paper, we may assume that, unless explicitly stated otherwise, matroids are given as \emph{basis oracles}.
In this model, we can access a matroid $M = (E, \mathcal B)$ through an oracle that decides whether $X \in \mathcal B$ for given $X \subseteq E$.\footnote{\rev{Our algorithm also runs in polynomial time even when the input matroids are given as \emph{independence} or \emph{rank} oracles.}}
We also assume that we can evaluate this query in time $|E|^{O(1)}$.

Let $M_1 = (E_1, \mathcal B_1), \dots, M_k = (E_k, \mathcal B_k)$ be $k$ matroids and let $\mathbb M = (M_1, \dots, M_k)$.
For $i \in [k]$, let $B_i$ be a basis of $M_i$.
A tuple $\mathbb B = (B_1, \dots, B_k)$ of bases is called a \emph{basis sequence} of $\mathbb M$.
Since $E_i$ and $E_j$ may have an intersection for distinct $i$ and $j$, $B_i$ and $B_j$ are not necessarily disjoint.
We say that $\mathbb B$ is \emph{feasible} if $B_i \cap B_j = \emptyset$ for distinct $i,j \in [k]$.
For two feasible basis sequences $\mathbb B = (B_1, \dots, B_k)$ and $\mathbb B' = (B'_1, \dots, B'_k)$ of $\mathbb M$, we say that $\mathbb B$ is \emph{adjacent} to $\mathbb B'$ if there is an index $i \in [k]$ such that $B_j = B'_j$ for $j \in [k]\setminus \{i\}$ and $B'_i = B_i - x + y$ for some $x \in B_i$ and $y \in E_i \setminus B_i$.
A \emph{reconfiguration sequence} between $\mathbb B$ and $\mathbb B'$ is a tuple of feasible basis sequences $\rseq{\mathbb B_0, \mathbb B_1, \dots, \mathbb B_\ell}$ such that $\mathbb B_0 = \mathbb B$, $\mathbb B_\ell = \mathbb B'$, and $\mathbb B_{i-1}$ and $\mathbb B_i$ are adjacent for all $i \in [\ell]$.
The \emph{length} of the reconfiguration sequence is defined as $\ell$.

Let $M = (E, \mathcal B)$ be a matroid.
The \emph{dual} of $M$ is a pair $M^* = (E, \{E \setminus B \mid B \in \mathcal B\})$, which also forms a matroid~\cite{Oxley:Matroid:2006}.
A \emph{coloop} of a matroid $M$ is an element $e \in E$ that belongs to all the bases of $M$, that is, $e \in B$ for all $B \in \mathcal B$.
Let $M = (E, \mathcal B)$ and $M' = (E', \mathcal B')$ be matroids and 
let $\mathcal B^*$ be the family of maximal sets in $\{B \cup B' \mid B \in \mathcal B, B' \in \mathcal B'\}$.
Then, the pair $(E \cup E', \mathcal B^*)$ is the \emph{matroid union} of $M$ and $M'$, which is denoted $M \vee M'$.
It is well known that $M \vee M'$ is also a matroid~\cite{Oxley:Matroid:2006}.
We can generalize this definition for more than two matroids: For $k$ matroids $M_1, \dots, M_k$, the matroid union of $M_1, \dots, M_k$ is denoted by $\bigvee_{i=1}^k M_i$.
\rev{If the ground sets $E$ and $E'$ of $M$ and $M'$ are disjoint, then $M \vee M'$ is called the \emph{direct sum} of $M$ and $M'$, and we write $M \oplus M'$ instead of $M \vee M'$.}

In our proofs, we use certain matroids.
Let $E$ be a finite set.
For an integer $r$ with $0 \le r \le |E|$, the \emph{rank-$r$ uniform matroid} on $E$ is the pair $(E, \{B \subseteq E \mid |B| = r\})$, that is, the set of bases consists of all size-$r$ subsets of $E$.
Let $\{E_1, \dots, E_k\}$ be a partition of $E$ (i.e., $E = \bigcup_{i=1}^k E_i$ and $E_i \cap E_j = \emptyset$ for distinct $i,j \in [k]$).
For each $i \in [k]$, we set $r_i$ as an integer with $0 \le r_i \le |E_i|$.
\rev{If $\mathcal{B} \subseteq 2^E$ consists of the sets $B$ satisfying $|B \cap E_i| = r_i$ for each $i \in [k]$,
then the pair $(E, \mathcal B)$ forms a matroid, called the \emph{partition matroid}.}
We can construct such a partition matroid by taking the \rev{direct sum} of the rank-$r_i$ uniform matroid\rev{s} on $E_i$ for $i$.

Let $D = (V, A)$ be a directed graph.
For an arc $a \in A$, we write $\head{a}$ to denote the head of $e$ and $\tail{a}$ to denote the tail of $e$.
A \emph{matching} of $D$ is a set $N \subseteq A$ of arcs such that no pair of arcs in $N$ share a vertex. 
A \emph{walk} in $D$ is a sequence $(v_0, a_1, v_1, a_2, \dots, a_\ell, v_\ell)$ such that $\tail{a_i} = v_{i-1}$ and $\head{a_i} = v_{i}$ for all $i \in [\ell]$.
When no confusion is possible, we may identify the directed graph with its arc set.

\section{Polynomial-time algorithm}

This section is devoted to a polynomial-time algorithm for \prb{\ourprob},
implying \Cref{thm:solvability}.
Let $M_1 = (E_1, \mathcal{B}_1), M_2 = (E_2, \mathcal{B}_2), \dots, M_k = (E_k, \mathcal{B}_k)$ be $k$ matroids that are given as basis oracles.
We denote by $\mathbb{M} = (M_1, M_2, \dots, M_k)$ the tuple of matroids $M_1, \dots, M_k$.

Let $\mathbb{B} = (B_1, \ldots, B_k)$ and $\mathbb{B}' = (B_1', \ldots, B_k')$ be two feasible basis sequences of $\mathbb{M}$.
Take any coloop $x$ of the matroid union $\bigvee_{i=1}^k M_i$.
Since all bases in $\mathbb B$ are mutually disjoint, $\bigcup_{i=1}^k B_i$ is a basis of $\bigvee_{i=1}^k M_i$.
This implies that $x \in B_i$ for some $i$.
Suppose that there is a feasible basis sequence $(B_1, \dots, B_{i-1}, B_i - x + y, B_{i+1}, \dots, B_k)$
of $\mathbb M$ obtained from $\mathbb B$ by exchanging $x \in B_i$ with $y \in E_i \setminus B_i$ in $M_i$.
As it is feasible, $\bigcup_{i=1}^k B_i - x + y$ is also a basis of $\bigvee_{i=1}^k M_i$, contradicting the fact that $x$ is a coloop.
This implies that every coloop in $\bigvee_{i=1}^k M_i$ belongs to a basis in a feasible basis sequence that is reconfigurable from $\mathbb B$.
More formally, let $K$ denote the set of coloops in $\bigvee_{i = 1}^k M_i$.
If $\mathbb B$ is reconfigurable to $\mathbb B'$, we have $(K \cap B_1, \dots, K \cap B_k) = (K \cap B'_1, \dots, K \cap B'_k)$.
The following theorem says that this necessary condition is also sufficient.
\begin{theorem}\label{thm:reconfigure}
    Let $K$ be the set of coloops of $\bigvee_{i = 1}^k M_i$.
    For feasible basis sequences $\mathbb{B} = (B_1, \ldots, B_k)$ and $\mathbb{B}' = (B_1', \ldots, B_k')$ of $\mathbb{M}$,
    one is reconfigurable to the other if and only if $(K \cap B_1, \ldots, K \cap B_k) = (K \cap B'_1, \ldots, K \cap B'_k)$.
\end{theorem}

The proof of \Cref{thm:reconfigure} is given in \Cref{subsec:proof} below.
Before the proof, we introduce the concept of \emph{exchangeability graphs}
and present its properties in \Cref{subsec:exchangeability}.

\subsection{Exchangeability graph}\label{subsec:exchangeability}
For a matroid $M = (E, \mathcal{B})$ and a basis $B \in \mathcal{B}$,
the \emph{exchangeability graph} of $M$ with respect to $B$,
denoted as $D(M, B)$,
is a directed graph whose vertex set is the ground set $E$ of $M$
and whose arc set $A$ is
\begin{align*}
    A := \{ (x, y) \mid x \in B \text{ and } y \in E \setminus B \text{ such that }B -x + y \in \mathcal{B} \}.
\end{align*}
Note that $D(M, B)$ is bipartite;
all arcs go from $B$ to $E \setminus B$.

Let $N = \{ (x_1, y_1), (x_2, y_2), \dots, (x_n, y_n) \} \subseteq A$ be a matching of $D(M, B)$ and
let $B \symdif N := B \setminus \{ x_1, x_2, \dots, x_n \} \cup \{ y_1, y_2, \dots, y_n \}$.
We say that $N$ is \emph{unique} if there is no perfect matching $N'$ other than $N$ in the subgraph of $D(M, B)$ induced by $\{x_1, \dots, x_n, y_1, \dots, y_n\}$.
The following is a well-known lemma in matroid theory, called the \emph{unique-matching lemma}.
\begin{lemma}[e.g., {\cite[Lemma~2.3.18]{Murota2010-ar}}]\label{lem:unique-matching}
If $N = \{ (x_1, y_1), (x_2, y_2), \dots, (x_n, y_n) \}$ is a unique matching in the subgraph of $D(M, B)$ induced by $\{ x_1, \dots, x_n, y_1, \dots, y_n \}$,
then $B \symdif N \in \mathcal{B}$.
\end{lemma}

The \emph{exchangeability graph} of $\mathbb{M}$ with respect to $\mathbb{B}$,
denoted as $D(\mathbb{M}, \mathbb{B})$,
is the union of the exchangeability graphs $D(M_i, B_i) = (E_i, A_i)$ of $M_i$ with respect to $B_i$ for all $i \in [k]$.
In the following, the vertex set of $D(\mathbb M, \mathbb B)$ is denoted by $E$, that is, $E = \bigcup_{i=1}^k E_i$.
We note that, for distinct $i, j \in [k]$,
the two arc sets $A_i$ and $A_j$ are disjoint,
since $B_i \cap B_j = \emptyset$.
A walk $W$ in $D(\mathbb{M}, \mathbb{B})$ is called a \emph{tadpole-walk}
if $W$ is of the form
\begin{align}\label{eq:W}
    (x_0, a_1, x_1, \ldots, x_{m-1}, a_{m}, x_m = x_0, a_{m+1}, x_{m+1}, a_{m+2}, \ldots, a_{n}, x_n)
\end{align}
for some $0 \leq m < n$
such that the former part $(x_0, a_1, x_1, \ldots, x_{m-1}, a_{m}, x_m = x_0)$ forms a directed cycle and the latter part $(x_m = x_0, a_{m+1}, x_{m+1}, \ldots, x_n)$ forms a directed path with $x_n \in E \setminus \bigcup_{i = 1}^k B_i$,
where $x_0, x_1, \dots, x_n$ are distinct except for $x_0 = x_m$ if $m > 0$.
See Figure~\ref{fig:tadpole} for an illustration.
\begin{figure}
    \centering
    \includegraphics[width=0.55\textwidth]{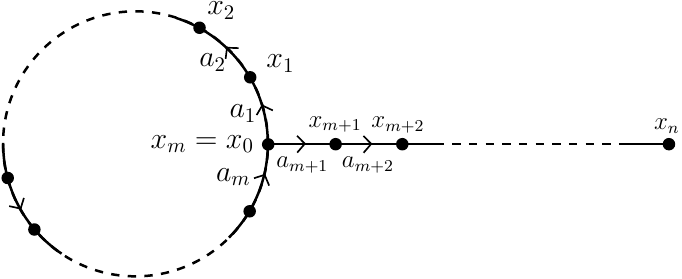}
    \caption{A tadpole-walk starting from $x_0$ and ending at $x_n$.}
    \label{fig:tadpole}
\end{figure}
The former part can be empty; in this case, $W$ is just a directed path
ending at some vertex in $E \setminus \bigcup_{i = 1}^k B_i$.
We introduce a total order $\prec$ on the vertex set of $W$ as:
The smallest vertex is $x_0 (= x_m)$
and
$x_i \prec x_j$ if and only if $i < j$
for other vertices $x_i, x_j$.
We say that $W$ is \emph{shortcut-free} if, for all $i \in [k]$ and two arcs $a, a' \in W \cap A_i$ with $\tail{a} \prec \tail{a'}$,
we have $(\tail{a} , \head{a'}) \notin A_i$.
A subgraph $W'$ of $D(\mathbb{M}, \mathbb{B})$ is said to be \emph{valid}
if it is the disjoint union of a (possibly empty) directed path ending at some vertex in $E \setminus \bigcup_{i = 1}^k B_i$ and a (possibly empty) directed cycle.
For a valid subgraph $W'$ of $D(\mathbb{M}, \mathbb{B})$,
we define
$\mathbb{B} \symdif W' := \left(B_1 \symdif (W' \cap A_1), B_2 \symdif (W' \cap A_2), \dots, B_k \symdif (W' \cap A_k)\right)$.
Observe first that $W' \cap A_i$ forms a matching in $D(M_i, B_i)$ for each $i$.
To see this, suppose that there are two arcs $a$ and $a'$ in $W' \cap A_i$ that share a vertex $x$.
Since each component of $W'$ is either a directed path or a directed cycle, we can assume that $\head{a} = \tail{a'} = x$.
However, $x \notin B_i$ as $\head{a} = x$ and $x \in B_i$ as $\tail{a'} = x$, a contradiction.
Observe next that $|\bigcup_{i=1}^k B_i| = |\bigcup_{i=1}^k (B_i \symdif (W' \cap A_i))|$.
This follows from the fact that the path component has a sink vertex in $E \setminus \bigcup_{i = 1}^k B_i$ (if it is nonempty).

The following two lemmas play important roles in the proof of \Cref{thm:reconfigure}.
\begin{lemma} \label{lem:feasibility}
    Suppose that $W$ is a shortcut-free tadpole-walk in $D(\mathbb M, \mathbb B)$
    and $W'$ is a valid subgraph of $W$.
    Then $\mathbb{B} \symdif W'$ is a feasible basis sequence of $\mathbb M$.
\end{lemma}
\begin{proof}
    We first observe that $B_i \symdif (W' \cap A_i)$ and $B_j \symdif (W' \cap A_j)$ are disjoint for distinct $i,j \in [k]$.
    This follows from the following facts: $|B_i| = |B_i \symdif (W' \cap A_i)|$ for each $i$ and $|\bigcup_{i=1}^k B_i| = |\bigcup_{i=1}^k (B_i \symdif (W' \cap A_i))|$.
    Thus, it suffices to show that $B_i \symdif (W' \cap A_i) \in \mathcal{B}_i$ for each $i$.

    Let $b_1, b_2, \dots, b_\ell$ be the arcs in the matching $W' \cap A_i$;
    we may assume that $i < j$ if and only if $\tail{b_i} \prec \tail{b_j}$.
    Since $W$ is shortcut-free,
    we have $(\tail{b_i}, \head{b_j}) \notin A_i$ for any distinct $i,j \in [\ell]$ with $i < j$.
    Observe that $W' \cap A_i$ forms a unique matching in $D(M_i, B_i)$.
    This can be seen by considering the other case that $W' \cap A_i$ is not unique in $D(M_i, B_i)$, yielding that $D(M_i, B_i)$ has an arc $(\tail{b_i}, \head{b_j})$ for some $i, j \in [\ell]$ with $i < j$.
    Thus $B_i \symdif (W' \cap A_i) \in \mathcal{B}_i$ by \Cref{lem:unique-matching}.
\end{proof}

\begin{lemma}\label{lem:coloop}
    Let $\mathbb{B} = (B_1, \ldots, B_k)$ be a feasible basis sequence of $\mathbb{M}$
    and $B := \bigcup_{i=1}^k B_i$.
    For $y \in E \setminus B$,
    we denote by $T_y$ the set of vertices that are reachable to $y$ in $D(\mathbb{M}, \mathbb{B})$, \rev{that is, the set of vertices $x$ in $E$ such that there is a directed path from $x$ to $y$ in $D(\mathbb{M}, \mathbb{B})$}.
    Then the set of coloops of $M \coloneqq \bigvee_{i=1}^kM_i$ is equal to $B \setminus \bigcup_{y \in E \setminus B} T_y$.
\end{lemma}
\begin{proof}
    Clearly $B$ contains all coloops of $M$ as $B$ is a basis of $M$.
    By considering the basis exchange axiom for the dual matroid $M^*$ of $M$,
    an element $x \in B$ is not a coloop of $M$ if and only if there is $y \in E \setminus B$ such that $B - x + y$ is a basis of $M$.
    Here, it easily follows from~\cite[Theorem~42.4]{book/Schrijver03} that,
    for $x \in B$ and $y \in E \setminus B$,
    the set $B - x + y$ is a basis of $M$
    if and only if there is a directed path from $x$ to $y$ in $D(\mathbb{M}, \mathbb{B})$.
    Hence the existence of such $y \in E \setminus B$
    can be rephrased as
    the existence of a directed path from $x$ to some vertex $y \in E \setminus B$ in $D(\mathbb{M}, \mathbb{B})$.
    This implies that
    the set of coloops of $M$ is equal to $B \setminus \bigcup_{y \in E \setminus B} T_y$,
    where $T_y$ denotes the set of vertices that are reachable to $y$ in $D(\mathbb{M}, \mathbb{B})$.
\end{proof}

Using~\Cref{lem:coloop}, we can decide in polynomial time whether the condition $(K \cap B_1, \ldots, K \cap B_k) = (K \cap B'_1, \ldots, K \cap B'_k)$ in \Cref{thm:reconfigure} holds as follows.
Let $E = \bigcup_{i=1}^k E_i$.
We can construct the exchangeability graph $D(\mathbb{M}, \mathbb{B})$ with $\sum_{i=1}^k |E_i|^2 \le k|E|^2$ oracle calls.
By~\Cref{lem:coloop}, we can compute the set $K$ of coloops of $M$ in time $O(|E|^2)$ using a standard graph search algorithm.

\subsection{Proof of \texorpdfstring{\Cref{thm:reconfigure}}{}}\label{subsec:proof}
In this subsection,
we provide the proof of \Cref{thm:reconfigure},
and then we also see that \Cref{thm:solvability} follows from our proof of \Cref{thm:reconfigure}.

We define the distance $d(\mathbb{B}, \mathbb{B}')$ between $\mathbb{B}$ and $\mathbb{B}'$ by
$d(\mathbb{B}, \mathbb{B}') := \sum_{i = 1}^k |B_i \symdif B_i'|$.
As we have already seen the only-if part of \Cref{thm:reconfigure} in the previous subsection, 
in the following, we show the if part by induction on $d(\mathbb{B}, \mathbb{B}')$.

It is easy to see that $d(\mathbb{B}, \mathbb{B}') = 0$ if and only if $\mathbb{B} = \mathbb{B}'$.
Suppose that $d(\mathbb{B}, \mathbb{B}')>0$.
If there is a feasible basis sequence $\mathbb{B}'' = (B_1'', B_2'', \dots, B_k'')$ of $\mathbb M$ such that
$\mathbb{B}$ is reconfigurable to $\mathbb{B}''$ and $d(\mathbb{B}'', \mathbb{B}') < d(\mathbb{B}, \mathbb{B}')$,
we have $(B_1' \cap K, \dots, B_k' \cap K) = (B_1 \cap K, \dots, B_k \cap K) = (B_1'' \cap K, \dots, B_k'' \cap K)$.
Hence $\mathbb{B}''$ is reconfigurable to $\mathbb{B}'$ by induction,
which implies that $\mathbb{B}$ is reconfigurable to $\mathbb{B}'$.
Thus, our goal is to find such a feasible basis sequence $\mathbb{B}''$.
To this end, we first compute a shortcut-free tadpole-walk $W$ in $D(\mathbb M, \mathbb B)$ and then transform $\mathbb B$ to $\mathbb B''$ one-by-one along this $W$.
A crucial observation in this transformation is that each intermediate \rev{basis sequence} is of the form $\mathbb B \symdif W'$ for some valid subgraph $W'$ of $W$, meaning that it is a feasible basis sequence of $\mathbb M$ by \Cref{lem:feasibility}.

Take any $x_0 \in \bigcup_{i = 1}^{k} B_i \setminus B'_i$, say, $x_0 \in B_{i_0} \setminus B'_{i_0}$.
Then there is $x_1 \in B'_{i_0} \setminus B_{i_0}$ such that $B_{i_0} - x_0 + x_1 \in \mathcal{B}_{i_0}$.
Hence we have $a_1 = (x_0, x_1) \in A_{i_0}$.
If $x_1 \in E \setminus \bigcup_{i = 1}^k B$, 
we obtain a tadpole-walk $(x_0, a_1, x_1)$;
we are done.
Otherwise, this vertex $x_1$ belongs to $B_{i_1}$ for some $i_1 \ (\neq i_0)$.
In particular, by $x_1 \in B'_{i_0}$, we have $x_1 \in B_{i_1} \setminus B'_{i_1}$.
Hence there is $x_2 \in B'_{i_1} \setminus B_{i_1}$ such that $B_{i_1} - x_1 + x_2 \in \mathcal{B}_{i_1}$, implying $a_2 = (x_1, x_2) \in A_{i_1}$.
By repeating this argument, we can find either of the following subgraphs of $D(\mathbb M, \mathbb B)$:
\begin{description}
    \item[Type~I:] a directed path $(x_0, a_1, x_1, \ldots, a_{n}, x_n)$ satisfying that $x_n \in E \setminus \bigcup_{i = 1}^k B_i$ and that $x_{\ell} \in B_{i_{\ell}} \setminus B'_{i_{\ell}}$, $x_{\ell + 1} \in B'_{i_{\ell}} \setminus B_{i_{\ell}}$, and $a_{\ell+1} \in A_{i_\ell}$ for all $\ell \in [0, n - 1]$.
    \item[Type~II:] a directed cycle $(x_p, a_{p+1}, x_{p+1}, \ldots, x_{q - 1}, a_{q}, x_{q} = x_p)$ satisfying that $x_{\ell} \in B_{i_{\ell}} \setminus B'_{i_{\ell}}$, $x_{\ell + 1} \in B'_{i_{\ell}} \setminus B_{i_\ell}$, and $a_{\ell+1} \in A_{i_{\ell}}$ for all $\ell \in [p, q - 1]$.
\end{description}
In the former case (Type~I), the resulting directed path is a tadpole-walk.
Consider the latter case (Type~II).
By the assumption that $(B_1 \cap K, \dots, B_k \cap K) = (B_1' \cap K, \dots, B_k' \cap K)$,
none of the vertices $x_p, \dots, x_q$ belongs to the set $K$ of coloops.
This implies that, by \Cref{lem:coloop}, $D(\mathbb M, \mathbb B)$ has a directed path from each vertex in the cycle to a vertex in $E \setminus \bigcup_{i = 1}^k B_i$.
We can choose such a directed path $(x_r, b_1, y_1, \ldots, b_m, y_m)$ from a vertex $x_r$ in the cycle to a vertex $y_m$ in $E \setminus \bigcup_{i = 1}^k B_i$ so that the path is arc-disjoint from the cycle, by taking a shortest one among all such paths.
Then, the walk $(x_r, a_r, x_{r+1}, \ldots, x_r, b_1, y_1, \ldots, b_m, y_m)$
forms a tadpole-walk.
We denote by $W$ the tadpole-walk obtained in these ways (Type~I and~II).
In the following,
by rearranging the indices, we may always assume that
$W$ is of the form~\eqref{eq:W},
where the former part $C = (x_0, a_1, x_1, \ldots, x_{m-1}, a_{m}, x_m = x_0)$ is a directed cycle and the later part $P = (x_m = x_0, a_{m+1}, x_{m+1}, \ldots, x_n)$ is a directed path in $D(\mathbb M, \mathbb B)$.
Note that the directed cycle $C$ can be empty, which corresponds to Type~I.

We next update the above $W$ so that $W$ becomes shortcut-free.
Suppose that $W$ is not shortcut-free.
Then there are arcs $a, a' \in W \cap A_i$ such that $\tail{a} \prec \tail{a'}$ satisfying $a'' \coloneqq (\tail{a}, \head{a'}) \in A_i$.
Let $a = (x_p, x_{p+1})$ and let $a' = (x_q, x_{q + 1})$.
Since $W \cap A_i$ is a matching in $D(\mathbb M, \mathbb B)$, we have $p + 1 \neq q$.
We then execute one of the following update procedure:
\begin{itemize}
    \item If $a$ and $a'$ belong to the directed path $P$,
    then update $W$ as 
    \begin{align*}
        W \leftarrow (x_0, \ldots, x_{m-1}, a_{m}, x_m = x_0, a_{m+1}, \dots, a_{p}, x_p, a'', x_{q+1}, \ldots, x_n).
    \end{align*}
    \item If $a$ and $a'$ belong to the directed cycle $C$,
    then update $W$ as
    \begin{align*}
        W \leftarrow (x_0, a_1, \ldots, a_p, x_p, a'', x_{q+1}, \dots, x_m = x_0, a_{m+1}, x_{m+1}, \ldots, x_n).
    \end{align*}
    \item If $a$ belongs to $C$ and $a'$ belongs to $P$,
    then update $W$ as
    \begin{align*}
        W \leftarrow (x_p, a_{p+1}, \dots, a_p, x_p, a'', x_{q+1}, \dots, x_n).
    \end{align*}
\end{itemize}
See~\Cref{fig:tadpole-t123} for illustrations.
\begin{figure}
    \centering
    \includegraphics[width=0.9\textwidth]{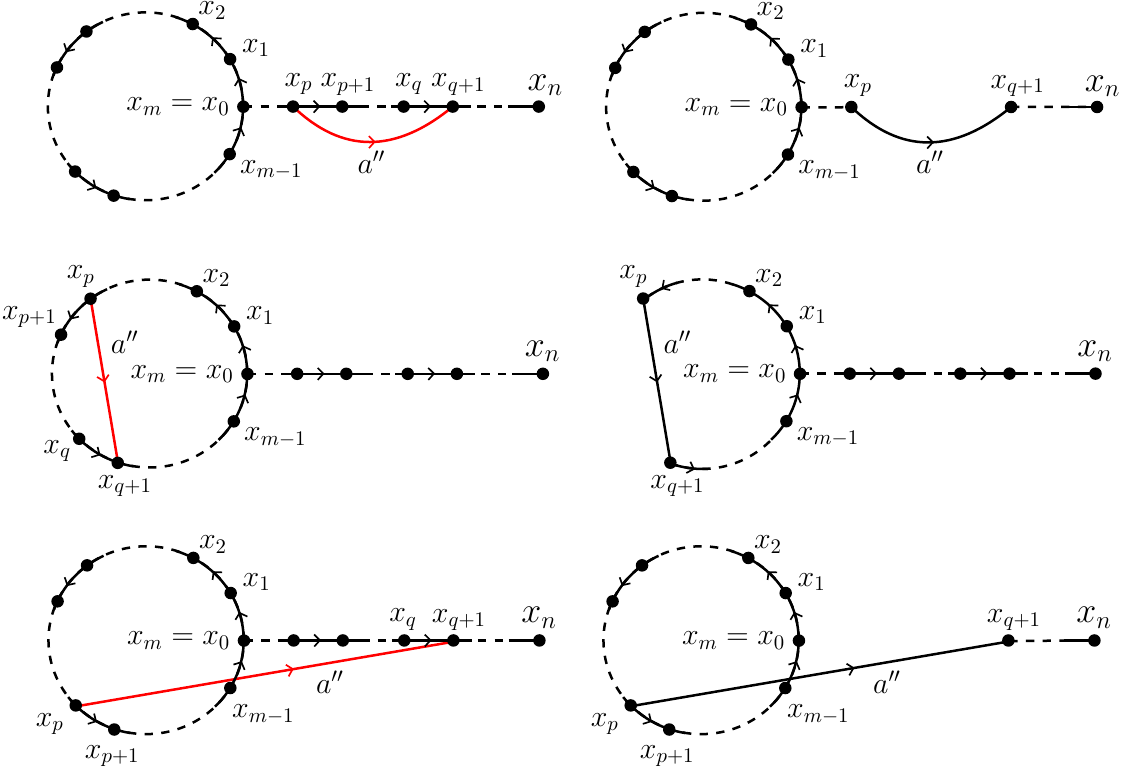}
    \caption{The figure depicts tadpole-walks (with shortcuts) and their updated tadpole-walks. }
    \label{fig:tadpole-t123}
\end{figure}

Suppose that $W$ is a tadpole-walk of Type~I.
In this case, the second and third cases never occur.
By the choice of $a = (x_p, x_{p+1})$, $a' = (x_q, x_{q+1})$, and $a'' = (x_p, x_{q+1})$,
we have $x_p \in B_i \setminus B_i'$, $x_{q+1} \in B_i' \setminus B_i$, and $(x_p, x_{q+1}) \in A_i$.
Moreover, the updated $W$ is a directed path ending at $x_n \in E \setminus \bigcup_{i = 1}^k B_i$, which implies that $W$ is still a tadpole-walk of Type~I.
Suppose next that $W$ is of Type~II.
In the first and third cases, the cycle part does not change; the updated $W$ is still of a tadpole-walk Type~II.
In the second case, the cycle part is shortened by $a''$ but the updated $W$ is still a tadpole-walk as well.
By the choice of $a$, $a'$, and $a''$,
we have $x_p \in B_i \setminus B_i'$, $x_{q+1} \in B_i' \setminus B_i$, and $(x_p, x_{q+1}) \in A_i$.
Hence the resulting $W$ is still a tadpole-walk of Type~II.
Since this update procedure strictly reduces the size of $W$,
we can eventually obtain a shortcut-free tadpole-walk in polynomial time.

Finally, we construct a reconfiguration sequence based on a shortcut-free tadpole-walk $W$ of each type.
Suppose that $W$ is of Type~I, i.e.,
$W = (x_0, a_1, x_1, \ldots, x_n)$ is a directed path with $n \geq 1$.
For $p \in [n-1]$,
let $W_p$ denote the subgraph of $W$ induced by
$\{ a_{n-p+1}, \dots, a_n \}$.
Then $W_p$ forms a directed path
$(x_{n-p}, a_{n-p+1}, x_{n-p+1}, \dots, x_n)$,
which implies that $W_p$ is valid.
By \Cref{lem:feasibility},
$\mathbb{B} \symdif W_p$ is a feasible basis sequence for each $p$.
Furthermore,
we have
$\mathbb{B} \symdif W_p = (\mathbb{B} \symdif W_{p-1}) \symdif (x_{n-p+1}, x_{n-p})$
(in which $(x_{n-p+1}, x_{n-p})$, the reverse of $a_{n-p+1}$, can be viewed as an arc in $D(\mathbb{M}, \mathbb{B} \symdif W_{p-1})$).
This implies that $\mathbb{B} \symdif W_{p-1}$ and $\mathbb{B} \symdif W_{p}$ are adjacent for all $p \in [n-1]$,
where $W_0 := \emptyset$.
Hence
\begin{align*}
    \langle \mathbb{B} = \mathbb{B} \symdif W_0, \mathbb{B} \symdif W_1, \mathbb{B} \symdif W_2, \dots, \mathbb{B} \symdif W_{n-1} = \mathbb{B} \symdif W \rangle
\end{align*}
is a reconfiguration sequence from $\mathbb B$ to $\mathbb B \symdif W$.
In addition, since $x_{\ell} \in B_{i_\ell} \setminus B_{i_\ell}'$ and $x_{\ell + 1} \in B_{i_\ell}' \setminus B_{i_\ell}$ for each $\ell \in [0, n-1]$,
we have $d(\mathbb{B} \symdif W, \mathbb{B}') = d(\mathbb{B}, \mathbb{B}') - 2n < d(\mathbb{B}, \mathbb{B}')$.

Suppose next that $W$ is of Type~II, i.e.,
$W$ is of the form~\eqref{eq:W}
with $0 < m < n$,
where the (nonempty) former part $C = (x_0, a_1, x_1, \ldots, x_{m-1}, a_{m}, x_m = x_0)$ is a directed cycle and the later part $P = (x_m = x_0, a_{m+1}, x_{m+1}, \ldots, x_n)$ is a directed path.
For $p \in [n-1]$,
let $W_p$ denote the subgraph of $W$ induced by
$\{ a_{n-p+1}, \dots, a_n \}$,
which forms a directed path $(x_{n-p}, a_{n-p+1}, x_{n-p+1}, \dots, x_n)$ as in the case of Type~I
and is valid.
For $p \in [n, 2n-m-1]$,
let $W_p$ denote the subgraph of $W$ induced by
$\{ a_1, a_2, \dots, a_m \} \cup \{ a_{(m -n+2) + p}, a_{(m -n+2) + (p+1)}, \dots, a_n \}$,
where $W_{2n-m-1}$ is defined as $C$.
In this case, $W_p$ forms the disjoint union of the directed cycle $C$ and the subpath of $P$ starting from $x_{m-n+1+p}$ ending at $x_n \in E \setminus \bigcup_{i = 1}^k B_i$;
the subpath is empty if $p = 2n-m-1$.
Thus $W_p$ is also valid.
By \Cref{lem:feasibility},
$\mathbb{B} \symdif W_p$ is feasible for each $p \in [2n-m-1]$.
Furthermore, we have
\begin{align*}
    \mathbb{B} \symdif W_{p} =
    \begin{cases}
        (\mathbb{B} \symdif W_{p-1}) \symdif (x_{n-p+1}, x_{n-p}) & \text{if $p \in [n-1]$},\\
        (\mathbb{B} \symdif W_{p-1}) \symdif (x_{m+1}, x_1) & \text{if $p = n$},\\
        (\mathbb{B} \symdif W_{p-1}) \symdif a_{(m -n+1) + p} & \text{if $p \in [n+1, 2n-m-1]$}.
    \end{cases}
\end{align*}
See Figure~\ref{fig:tadpole-p=n} for the case of $p = n$.
\begin{figure}
    \centering
    \includegraphics[width=0.45\textwidth]{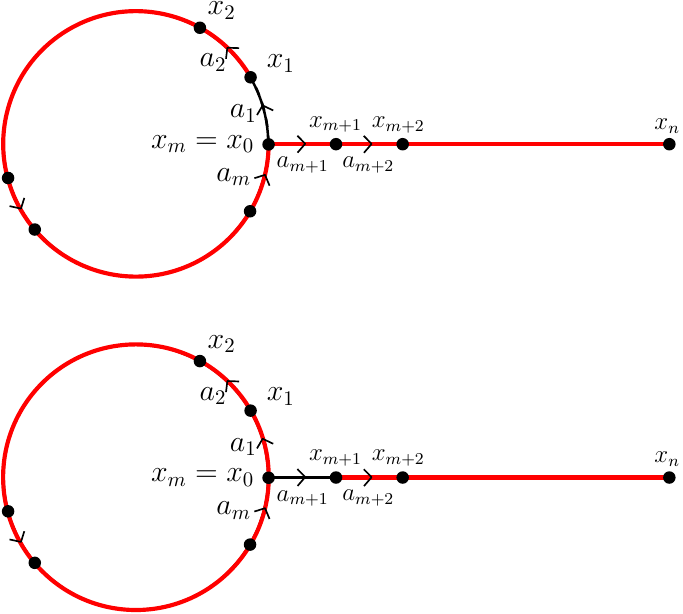}
    \caption{The bold red walk in the upper digraph represents $W_p$ for $p = n-1$,
    and
    that in the lower digraph for $p = n$.}
    \label{fig:tadpole-p=n}
\end{figure}
This implies that $\mathbb{B} \symdif W_{p-1}$ and $\mathbb{B} \symdif W_{p}$ are adjacent for all $p \in [2n-m-1]$,
where $W_0 := \emptyset$.
Hence
\begin{align*}
    \langle \mathbb{B} = \mathbb{B} \symdif W_0, \mathbb{B} \symdif W_1, \mathbb{B} \symdif W_2, \dots, \mathbb{B} \symdif W_{2n- m -1} = \mathbb{B} \symdif C \rangle
\end{align*}
is a reconfiguration sequence from $\mathbb B$ to $\mathbb B \symdif C$.
In addition, since $x_{\ell} \in B_{i_\ell} \setminus B_{i_\ell}'$ and $x_{\ell + 1} \in B_{i_\ell}' \setminus B_{i_\ell}$ for each $\ell \in [m]$,
we have $d(\mathbb{B} \symdif C, \mathbb{B}') = d(\mathbb{B}, \mathbb{B}') - 2m < d(\mathbb{B}, \mathbb{B}')$.
This completes the proof of \Cref{thm:reconfigure}.

The above proof immediately turns into an algorithm for finding a feasible basis sequence $\mathbb{B}''$ with $d(\mathbb{B}'', \mathbb{B}') < d(\mathbb{B}, \mathbb{B}')$
in polynomial time.
As shown in the previous subsection, we can construct the exchangeability graph $D(\mathbb{M}, \mathbb{B})$ using $k|E|^2$ oracle calls.
We can compute a shortcut-free tadpole-walk in $D(\mathbb{M}, \mathbb{B})$ in $O(|E|^2)$ time.
Thus, we can compute a feasible basis sequence $\mathbb{B}''$ of $\mathbb{M}$ with $d(\mathbb{B}'', \mathbb{B}') < d(\mathbb{B}, \mathbb{B}')$ such that $\mathbb{B}$ is reconfigurable to $\mathbb{B}''$ in $O(|E|^2)$ time and $|E|^2$ oracle calls.
Since $d(\mathbb{B}, \mathbb{B}')$ is at most $2|E|$,
we can obtain an entire reconfiguration sequence from $\mathbb B$ to $\mathbb B'$ in $O(|E|^3)$ time and $|E|^3$ oracle calls in the case where $\mathbb{B}$ is reconfigurable to $\mathbb{B}'$.
Note that the length of the above reconfiguration sequence is $O(|E|^2)$.
Therefore, \Cref{thm:solvability} follows.

\section{Inapproximability of finding a shortest reconfiguration sequence}
In this section, we prove \Cref{thm:hardness}, that is, \prb{Shortest \ourprob} is hard to approximate in polynomial time under $\P \neq \NP$.
To show this inapproximability result, we perform a reduction from \prb{Set Cover}, which is notoriously hard to approximate.

Let $\mathcal{S} \subseteq 2^U$ be a family of subsets of a finite set $U$ where $n=|U|$ and $m=|\mathcal{S}|$. 
We say that a subfamily $\mathcal S' \subseteq \mathcal S$ \emph{covers} $U$ (or $\mathcal S'$ is a \emph{set cover} of $U$) if $U = \bigcup_{S \in \mathcal S'}S$.
\prb{Set Cover} is the problem that, given a set $U$ and a family $\mathcal S \subseteq 2^U$ of subsets of $U$, asks to find a minimum cardinality subfamily $\mathcal S' \subseteq \mathcal S$ that covers $U$.
\prb{Set Cover} is known to be hard to approximate: Raz and Safra~\cite{RazS97:sub-constant} showed that there is a constant $c^* > 0$ such that it is $\NP$-hard to find a $c^*\log (n+m)$-approximate solution of \prb{Set Cover}.
Throughout this section, we assume that the whole family $\mathcal S$ covers $U$.

From an instance $(U, \mathcal S)$ of \prb{Set Cover}, we construct two partition matroids $M_1 = (E_1, \mathcal B_1)$ and $M_2 = (E_2, \mathcal B_2)$ such that there is a set cover of $U$ of size at most $k$ if and only if there is a reconfiguration sequence between feasible basis sequences $\source{\mathbb{B}}$ and $\sink{\mathbb{B}}$ of $\mathbb M = (M_1, M_2)$ with length at most $\ell$ for some $\ell$. 

\subsection{Construction} \label{subsec:construction}
To construct the partition matroids $M_1$ and $M_2$, we use several uniform matroids and combine them into $M_1$ and $M_2$.
In the following, we assume that the sets in $\mathcal S$ are ordered in an arbitrary total order $\preceq$.
For each element $u \in U$, we define $f(u) + 3$ elements $\elmver{u}{1}, \elmver{u}{2}, \elmver{u}{3}, \cntver{u}{1}, \dots, \cntver{u}{f(u)}$ and three sets:
\begin{align*}
    \elmgrd{u}{1} \coloneqq \{\elmver{u}{1}, \elmver{u}{2}\},\qquad
    \elmgrd{u}{2} \coloneqq \{\elmver{u}{1}, \elmver{u}{2}, \elmver{u}{3}\},\qquad
    \elmgrd{u}{3} \coloneqq \{\elmver{u}{3}\}\cup \{ \cntver{u}{1}, \cntver{u}{2}, \ldots, \cntver{u}{f(u)}\},
\end{align*}
where $f(u) \coloneqq |\{S \in \mathcal{S} \mid u \in S\}|$ is the number of occurrences of $u$ in $\mathcal S$.
We denote by $\elmmat{u}{i}$ the rank-$1$ uniform matroid over $\elmgrd{u}{i}$ for $1 \le i \le 3$, that is, each basis of $\elmmat{u}{i}$ contains exactly one element in $\elmgrd{u}{i}$.
For each set $S \in \mathcal S$, we denote by $\setmat{S}{0}$ the uniform matroid of rank~$|S|$ with ground set $\setgrd{S}{0} \coloneqq \{\cntver{u}{\elemid{u}{S}} \mid u \in S\} \cup \{\setver{S}{1}\}$, where $\elemid{u}{S} = |\{S' \mid  S' \preceq S, u \in S'\}|$.
Note that $\setgrd{S}{0} \cap \setgrd{S'}{0} = \emptyset$ for distinct $S, S' \in \mathcal S$.
We let $L \coloneqq 2n^2$.
For $1 \le i \le L$, we denote by $\setmat{S}{i}$ the rank-$1$ matroid with ground set $\setgrd{S}{i} \coloneqq \{\setver{S}{i}, \setver{S}{i+1}\}$.
Then, we define two partition matroids $M_1$ and $M_2$ as:
\begin{align*}
    \rev{M_{1} \coloneqq \bigoplus_{u \in U} \elmmat{u}{1} \oplus \bigoplus_{u \in U} \elmmat{u}{3} \oplus \bigoplus_{S \in \mathcal{S}}\bigoplus_{i = 1}^{n^2} \setmat{S}{2i-1}, \qquad
    M_{2} \coloneqq \bigoplus_{u \in U} \elmmat{u}{2} \oplus \bigoplus_{S \in \mathcal{S}} \setmat{S}{0} \oplus \bigoplus_{S \in \mathcal{S}}\bigoplus_{i = 1}^{n^2} \setmat{S}{2i}.}
\end{align*}
The matroids $M_1$ and $M_2$ are illustrated in \Cref{fig:partition-matroids}.
\begin{figure}
    \centering
    \includegraphics[width=0.8\textwidth]{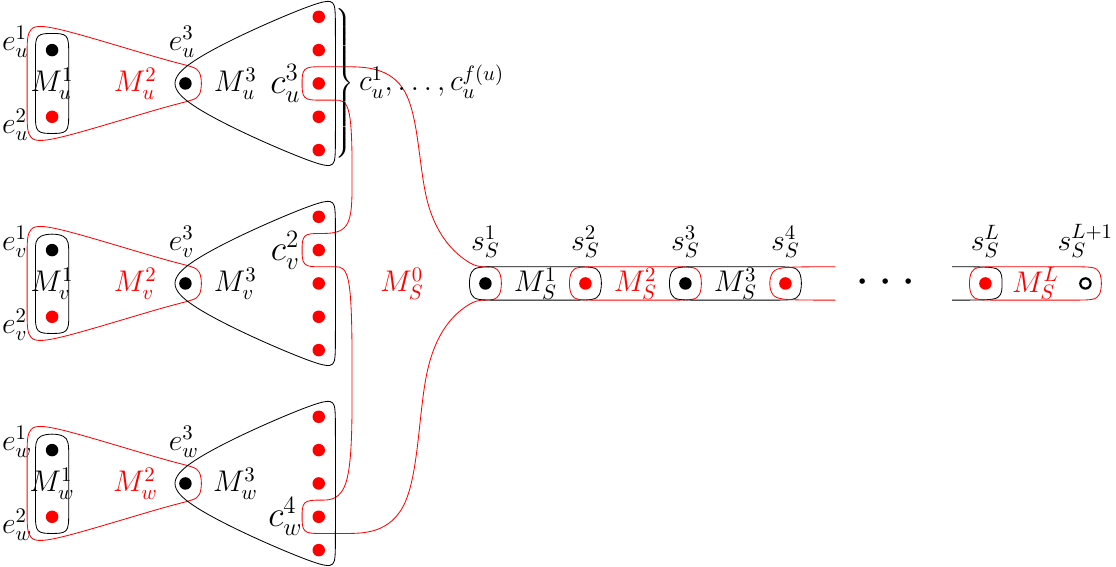}
    \caption{The figure depicts (hypergraph representations of) two partition matroids $M_1$ and $M_2$. A set $S \in \mathcal S$ contains three elements $u, v, w \in U$ with $\elemid{u}{S} = 3$, $\elemid{v}{S}= 2$, and $\elemid{w}{S} = 4$.
    Solid black circles represent elements in $\source{B}_1$, and solid red circles represent elements in $\source{B}_2$.
    }
    \label{fig:partition-matroids}
\end{figure}
We denote by $E_1$ and $E_2$ the ground sets and by $\mathcal B_1$ and $\mathcal B_2$ the collections of bases of $M_1$ and $M_2$, respectively.
Each uniform matroid constituting these partition matroids is called a \emph{block}.
Since $M_1$ and $M_2$ are partition matroids, the following observation follows.
\begin{observation}\label{obs:partition-matroid}
    Let $(B_1, B_2)$ be a feasible basis sequence of $(M_1, M_2)$ and let $x \in B_1$ be arbitrary.
    Then, for any element $y \in E_1 \setminus (B_1 \cup B_2)$ that belongs to the same block as $x$ in $M_1$, $(B_1 - x + y, B_2)$ is a feasible basis sequence of $(M_1, M_2)$.
    Similarly, let $x \in B_2$ be arbitrary.
    Then, for any element $y \in E_2 \setminus (B_1 \cup B_2)$ that belongs to the same block as $x$ in $M_2$, $(B_1, B_2  - x + y)$ is a feasible basis sequence of $(M_1, M_2)$.
\end{observation}

Let $\source{\mathbb{B}} = (\source{B}_1, \source{B}_2)$ be a feasible basis sequence such that
\begin{align*}
    \source{B}_1 &= \{\elmver{u}{1} \mid u\in U\} \cup \{\elmver{u}{3} \mid u \in U\} \cup \bigcup_{S \in \mathcal S}\{\setver{S}{2i-1} \mid i \in [\Llength]\},\\
    \source{B}_2 &= \{\elmver{u}{2} \mid u\in U\} \cup \bigcup_{S \in \mathcal S}\{\cntver{u}{\elemid{u}{S}} \mid u \in S\} \cup \bigcup_{S \in \mathcal S}\{\setver{S}{2i} \mid i \in [\Llength]\}.
\end{align*}
It is easy to verify that $\source{B}_1$ and $\source{B}_2$ are bases of $M_1$ and $M_2$, respectively.
Similarly, let $\sink{\mathbb{B}} = (\sink{B}_1, \sink{B}_2)$ be a feasible basis sequence such that
\begin{align*}
    \sink{B}_1 &= \{\elmver{u}{2} \mid u\in U\} \cup \{\elmver{u}{3} \mid u \in U\} \cup \bigcup_{S \in \mathcal S}\{\setver{S}{2i-1} \mid i \in [\Llength]\},\\
    \sink{B}_2 &= \{\elmver{u}{1} \mid u\in U\} \cup \bigcup_{S \in \mathcal S}\{\cntver{u}{\elemid{u}{S}} \mid u \in S\} \cup \bigcup_{S \in \mathcal S}\{\setver{S}{2i} \mid i \in [\Llength]\}.
\end{align*}
Let us note that $\setver{S}{L + 1} \notin \source{B}_1 \cup \source{B}_2 \cup \sink{B}_1 \cup \sink{B}_2$ for all $S \in \mathcal S$.
Moreover, we have $\source{B}_1 \setminus \sink{B}_1 = \sink{B}_2 \setminus \source{B}_2 = \{\elmver{u}{1} \mid u \in U\}$ and $\sink{B}_1 \setminus \source{B}_1 = \source{B}_2 \setminus \sink{B}_2 = \{\elmver{u}{2} \mid u \in U\}$.

\subsection{Correctness}
Before proceeding to our proof, we first give the intuition behind our construction.
Suppose that there are tokens on the elements in $\source{B}_1 \cup \source{B}_2$.
As observed in the previous subsection, we have $\elmver{u}{1} \in \source{B}_1 \setminus \sink{B}_1$ and $\elmver{u}{1} \in \sink{B}_2\setminus \source{B}_2$ for $u \in U$.
Moreover, $\elmver{u}{2} \in \sink{B}_1 \setminus \source{B}_1$ and $\elmver{u}{2} \in \source{B}_2\setminus \sink{B}_2$ for $u \in U$.
Thus, in order to transform $\source{\mathbb{B}}$ to $\sink{\mathbb{B}}$, we need to ``swap'' the tokens on $\elmver{u}{1}$ and $\elmver{u}{2}$.
However, as all the elements except for $\setver{S}{L+1}$ for $S \in \mathcal S$ are occupied by tokens in $\source{B}_1 \cup \source{B}_2$, this requires to move an ``empty space'' initially placed on $\setver{S}{L+1}$ to $\elmver{u}{3}$ for some $S \in \mathcal S$ with $u \in S$, and then swap the tokens on $\elmver{u}{1}$ and $\elmver{u}{2}$ using the empty space on $\elmver{u}{3}$.
By the construction of $M_1$ and $M_2$, this can be done by (1) shifting the tokens along the path between $\setver{S}{L+1}$ and $\setver{S}{1}$ one by one, (2) moving the empty space from $\setver{S}{1}$ to $\cntver{u}{\elemid{u}{S}}$, and then (3) moving the empty space from $\cntver{u}{\elemid{u}{S}}$ to $\elmver{u}{3}$, which requires at least $L$ exchanges.
As $L$ is sufficiently large, we need to cover the elements in $U$ with a small number of sets in $\mathcal S$ for a short reconfiguration sequence.
The following lemma gives an upper bound on the length of a shortest reconfiguration sequence.

\begin{lemma}\label{lem:exists_reconf}
    Let $\mathcal S^* \subseteq \mathcal S$ be a set cover of $U$ of size at most $k$.
    Then, there is a reconfiguration sequence between $\source{\mathbb{B}}$ and $\sink{\mathbb{B}}$ with length at most $2kL+ 7n$.
\end{lemma}

\begin{proof}
    Given a set cover $\mathcal S^* \subseteq \mathcal S$, we construct a reconfiguration sequence between $\source{\mathbb{B}}$ and $\sink{\mathbb{B}}$ by applying the algorithm described in Algorithm~\ref{alg:constructing-reconfiguration-sequence}.
    Let $\mathbb B = (B_1, B_2)$ be a feasible basis sequence of $(M_1, M_2)$.
    For $x, y \in E_1 \cup E_2$, we call $(x, y)$ a \emph{valid pair} if either
    \begin{itemize}\setlength{\leftskip}{0.3cm}
        \item[(1)] $x \in B_1$ and $y \in E_1 \setminus (B_1  \cup B_2)$ belong to the same block in $M_1$; or
        \item[(2)] $x \in B_2$ and $y \in E_2 \setminus (B_1 \cup B_2)$ belong to the same block in $M_2$.
    \end{itemize}
    For a valid pair $(x, y)$, we define
    \begin{align*}
        \mathbb B \symdif (x, y) = \begin{cases}
            (B_1 - x + y, B_2) & \text{if } (x, y) \text{ satisfies (1)},\\
            (B_1, B_2 - x + y) & \text{if } (x, y) \text{ satisfies (2)}.
        \end{cases}
    \end{align*}
    By~\Cref{obs:partition-matroid}, $\mathbb B \symdif (x, y)$ is a feasible basis sequence of $(M_1, M_2)$.
    \begin{algorithm}[t]
        \KwInput{A set cover $\mathcal S^* \subseteq \mathcal S$ of $U$.} 
        \KwOutput{A reconfiguration sequence between $\source{\mathbb{B}}$ and $\sink{\mathbb{B}}$.}
        $\Tilde{U} \gets \emptyset$, $\mathbb{B} \gets \source{\mathbb{B}}$\\
        \ForEach{$S \in \mathcal{S}^*$}{
            \For{$i = L, L - 1, \ldots, 1$}{
                $\mathbb{B} \gets \mathbb{B} \symdif (\setver{S}{i}, \setver{S}{i+1})$
            }
        }
        \ForEach{$S \in \mathcal{S}^*$}{
            \If{$S \setminus \Tilde{U} \neq \emptyset$}{
                \ForEach{$u \in S \setminus \Tilde{U}$}{
                    $\mathbb{B} \gets \mathbb{B} \symdif (\cntver{u}{\elemid{u}{S}}, \setver{S}{1})$\\ \label{alg:line:st}
                    $\mathbb{B} \gets \mathbb{B} \symdif (\elmver{u}{3}, \cntver{u}{\elemid{u}{S}})$\\
                    $\mathbb{B} \gets \mathbb{B} \symdif (\elmver{u}{2}, \elmver{u}{3})$\\
                    $\mathbb{B} \gets \mathbb{B} \symdif (\elmver{u}{1}, \elmver{u}{2})$\\
                    $\mathbb{B} \gets \mathbb{B} \symdif (\elmver{u}{3}, \elmver{u}{1})$\\
                    $\mathbb{B} \gets \mathbb{B} \symdif (\cntver{u}{\elemid{u}{S}}, \elmver{u}{3})$\\
                    $\mathbb{B} \gets \mathbb{B} \symdif (\setver{S}{1}, \cntver{u}{\elemid{u}{S}})$ \label{alg:line:ed}
                }
                $\Tilde{U} \gets \Tilde{U} \cup S$
            }
        }
        \ForEach{$S \in \mathcal{S}^*$}{
            \For{$i = 1, 2, \ldots, L$}{
                $\mathbb{B} \gets \mathbb{B} \symdif (\setver{S}{i+1}, \setver{S}{i})$
            }
        }
        \caption{An algorithm for constructing a reconfiguration sequence between $\source{\mathbb{B}}$ and $\sink{\mathbb{B}}$ from a set cover $\mathcal S^* \subseteq \mathcal S$ of $U$.}\label{alg:constructing-reconfiguration-sequence}
    \end{algorithm}

    When we update a feasible basis sequence $\mathbb B = (B_1, B_2)$ with $\mathbb B \symdif (x, y)$ for some $x, y \in E_1 \cup E_2$ in the algorithm, the pair $(x, y)$ is always assured to be valid.
    Thus, all the pairs $\mathbb B = (B_1, B_2)$ appearing in the execution of the algorithm are feasible basis sequences of $(M_1, M_2)$.
    Since $\mathcal S^*$ is a set cover of $U$, we have $\tilde U = U$ when the algorithm terminates.
    Thus, for each $u \in U$, the steps from line~\ref{alg:line:st} to line~\ref{alg:line:ed} are executed exactly once.
    This implies that the algorithm correctly computes a reconfiguration sequence between $\source{\mathbb{B}}$ and $\sink{\mathbb{B}}$ with length $2kL + 7n$.
\end{proof}

\begin{lemma} \label{lem:exists_cover}
    Suppose that there is a reconfiguration sequence between $\source{\mathbb{B}}$ and $\sink{\mathbb{B}}$ of length $\ell$.
    Then, there is a set cover $\mathcal S^* \subseteq \mathcal S$ of $U$ with $|\mathcal{S}^*| \le \floor{\ell/2L}$.
\end{lemma}

\begin{proof}
    Let $\sigma = \rseq{\mathbb B_0, \dots, \mathbb B_\ell}$ be a reconfiguration sequence between $\source{\mathbb{B}}$ and $\sink{\mathbb{B}}$ of length $\ell$.
    For a feasible basis sequence $\mathbb B = (B_1, B_2)$, an element $e \in E_1 \cup E_2$ is said to be \emph{free} in $\mathbb B$ if $e \notin B_1 \cup B_2$.
    We define $\mathcal{S}^* \coloneqq \{S \in \mathcal S \mid \setver{S}{1} \text{ is free in } \mathbb B_i \text{ for some } i\}$.
    Then the following holds.
    \begin{claim}\label{cl:set_cover}
    The subfamily $\mathcal S^*$ of $\mathcal{S}$ is a set cover of $U$.
    \end{claim}
    \begin{proof}[Proof of \Cref{cl:set_cover}]
    Let $\mathbb{B}_i = (B_1^i, B_2^i)$ for $i \in [0, \ell]$.
    We first observe that, for $S \in \mathcal{S}$ and $i \in [0, \ell]$,
    if $s_S^1\in B^i_{1}$, then $\{ c_u^{\elemid{u}{S}} \mid u \in S \} \subseteq B^i_{2}$.
    This can be seen as follows.
    Since $s_S^1\in B^i_{1}$,
    we have $s_S^1 \notin B^i_{2}$.
    As $B^i_2$ must contain a basis $B_S^0$ of $M_S^0$,
    which is the uniform matroid of rank $|S|$ with the ground set $\{ c_u^{\elemid{u}{S}} \mid u \in S \} \cup \{s_S^1\}$,
    the basis $B_S^0$ must be $\{ c_u^{\elemid{u}{S}} \mid u \in S \}$.
    That is, we have $\{ c_u^{\elemid{u}{S}} \mid u \in S \} \subseteq B^i_2$.

    We then show the assertion of \Cref{cl:set_cover}.
    Suppose for contradiction that there is an element $u^* \in U$ that is not covered by $\mathcal S^*$.
    Then, for $S \in \mathcal S$ with $u^* \in S$, the element $\setver{S}{1}$ is not free in $\mathbb B_i$ for any $0 \le i \le \ell$,
    which implies that $\setver{S}{1}$ belongs to $B_1^i$.
    Thus, for each $i$, we have
    $B^i_2 \supseteq \bigcup_{S \in \mathcal S : u^{*} \in S}\{ c_u^{\elemid{u}{S}} \mid u \in S \} \supseteq \{ c_{u^*}^1, \dots, c_{u^*}^{f(u^*)} \}$,
    where the first inclusion follows from the above observation.
    By this inclusion with the fact that $M_{u^*}^3$ is the uniform matroid of rank $1$ with the ground set $\{e_{u^*}^3\} \cup \{ c_{u^*}^1, \dots, c_{u^*}^{f(u^*)} \}$,
    the basis $B_1^i$ must contain $e_{u^*}^3$ for each $i$.
    Hence, during the reconfiguration sequence $\sigma = \rseq{\mathbb B_0, \dots, \mathbb B_\ell}$,
    we cannot move any element in $E_{u^*}^1 = \{ e_{u^*}^1, e_{u^*}^2\}$ (or more precisely $E_{u^*}^1 \cup E_{u^*}^2 \cup E_{u^*}^3$).
    This contradicts that $\sigma$ is a reconfiguration sequence from $\source{\mathbb{B}}$ to $\sink{\mathbb{B}}$;
    recall $e_{u^*}^1 \in \source{B}_1 \setminus \sink{B}_1 = \sink{B}_2 \setminus \source{B}_2$ and $e_{u^*}^2 \in \sink{B}_1 \setminus \source{B}_1 = \source{B}_2 \setminus \sink{B}_2$.
    \end{proof}

In the reconfiguration sequence $\sigma = \rseq{\mathbb B_0, \dots, \mathbb B_\ell}$,
for each $S \in \mathcal{S}^*$,
the element $\setver{S}{L+1}$ must be free in $\mathbb B_0$ and $\mathbb B_\ell$,
and 
$\setver{S}{1}$ must be free at least once.
Hence, the length $\ell$ of $\sigma$ is at least $2L \cdot |\mathcal{S}^*|$,
where $L$ is equal to the number of required steps to move from a feasible basis sequence such that $\setver{S}{L+1}$ (resp. $\setver{S}{1}$) is free to another feasible basis sequence such that $\setver{S}{1}$ (resp. $\setver{S}{L+1}$) is free.
Since $\mathcal{S}^*$ is a set cover by \Cref{cl:set_cover},
we can conclude that there is a set cover of size at most $\floor{\ell/2L}$.
\end{proof}

\begin{proof}[Proof of \Cref{thm:hardness}]
    \rev{
    To prove the $\NP$-hardness, we give a polynomial-time reduction from \prb{Set Cover}.
    We claim that $I = (U, \mathcal{S})$ has a set cover of size at most $k$ if and only if there is a reconfiguration sequence between $\source{\mathbb{B}}$ and $\sink{\mathbb{B}}$ of length at most $(2k + 1) \cdot L$.
    We may assume $n \ge 4$.}
    
    \rev{
    Suppose that $I$ has a set cover of size at most $k$.
    Then, by \Cref{lem:exists_reconf} and $7n \le 2n^2 = L$, we can construct a reconfiguration sequence from $\source{\mathbb{B}}$ to $\sink{\mathbb{B}}$ of length at most $2kL + 7n \leq 2kL + 2n^2 = (2k + 1) \cdot L$, proving the forward implication.
    }
    
    \rev{
    Conversely, assume that there is a reconfiguration sequence between $\source{\mathbb{B}}$ and $\sink{\mathbb{B}}$ of length at most $(2k + 1) \cdot L$.
    Then, by \Cref{lem:exists_cover}, we obtain a set cover for $I$ of the size at most $\floor{(2k + 1) \cdot L/2L} = \floor{k + 1 / 2} = k$.
    }
    
    \rev{To prove the inapproximability, let $N = \sum_{i=1}^k |E_i|$ and} suppose that there exists a $c'\log N$-approximation algorithm $\mathcal{A}'$ for \prb{Shortest \ourprob} for some constant $c' > 0$. Then we construct an algorithm $\mathcal{A}$ that, given an instance $I = (U, \mathcal{S})$ of \prb{Set Cover}, outputs a set cover of $I$ as follows. 

    \begin{enumerate}
        \item Construct an instance $I' = (\mathbb M, \source{\mathbb{B}}, \sink{\mathbb{B}})$ of \prb{Shortest \ourprob} from an instance $I = (U, \mathcal S)$ of \prb{Set Cover} using the construction in \Cref{subsec:construction}.
        \item Compute a reconfiguration sequence $\sigma'$ of $I'$ by applying $\mathcal{A}'$.
        \item Compute a set cover $\mathcal{S}^*$ for $I$ from $\sigma'$ by \Cref{lem:exists_cover}.
    \end{enumerate}

    \begin{claim} \label{cl:approx_alg}
        For some constant $c > 0$, algorithm $\mathcal{A}$ produces a $c\log (n + m)$-approximation solution for \prb{Set Cover}.
    \end{claim}
    \begin{proof}[Proof of \Cref{cl:approx_alg}]
        Since $\mathcal S$ covers $U$, by~\Cref{lem:exists_reconf}, there is a reconfiguration sequence between $\source{\mathbb{B}}$ and $\sink{\mathbb{B}})$ of length at most $2L\cdot \OPT(I) + 7n$, where $\OPT(I)$ is the minimum cardinality of a set cover of $U$.
    Moreover, we have $N \le (n + m)^{d}$ for some constant $d$.
    Thus, $\mathcal{A}'$ outputs a reconfiguration sequence $\sigma'$ of length at most $\ell := c'\log N \cdot (2L\cdot\OPT(I) + 7n)$ in time $(n + m)^{O(1)}$. 
    Finally, by \Cref{lem:exists_cover}, we can compute a set cover $\mathcal S^* \subseteq \mathcal S$ of $U$ from $\sigma'$ with size at most $\ell/2L = c'\log N \cdot (\OPT(I) + o(1)) \leq 2 c'\log N \cdot \OPT(I)$.
    Since $N \le (n + m)^{d}$, we have 
    $|\mathcal{A}(I)| \le c\log(n + m) \cdot \OPT(I)$
    for any constant $c > 2 c'd$.
    \end{proof}

    By choosing the constant $c'$ as $c' < c^*/2d$, we derive a polynomial-time $c^*\log (n+m)$-approximation algorithm for \prb{Set Cover}, completing the proof of \Cref{thm:hardness}.
\end{proof}

\section{Conclusion}
In this paper, we studied \prb{\ourprob}, which is a generalization of \prb{Spanning Tree Sequence Reconfiguration}. For this problem, we first showed that \prb{\ourprob} can be solved in polynomial time, assuming that the input matroids are given as basis oracles.
Second, we showed that the shortest variant of \prb{\ourprob} is hard to approximate within a factor of $c \log n$ for some constant $c>0$ unless $\P = \NP$.

For future work, it is interesting to investigate the computational complexity of the special settings of \prb{\ourprob}. It would be interesting to design faster or simpler algorithms for \prb{\ourprob} with graphic matroids, that is, for \prb{Spanning Tree Sequence Reconfiguration}.
Our hardness result for the shortest variant uses two distinct partition matroids. Thus, it would be worth considering the case for two identical matroids. Finally, the computational complexity of \prb{Shortest Spanning Tree Sequence Reconfiguration} is another promising direction.

\section*{Acknowledgments}
The first author was supported by JSPS KAKENHI Grant Numbers JP21K17707, JP21H05852, JP22H00513, and JP23H04388.
The second author was supported by JSPS KAKENHI Grant Numbers JP22K17854, JP24K02901, JP24K21315.
The third author was supported by JSPS KAKENHI Grant Numbers JP20H00595, JP23K28034, JP24H00686, and JP24H00697.
The forth author was supported by JST SPRING Grant Number JPMJSP2125.
The fifth author was supported by JST SPRING Grand Number JPMJSP2114.

\providecommand{\noopsort}[1]{}

\end{document}